\documentclass[a4paper]{amsart}
\usepackage{amssymb}
\usepackage{amsthm}
\usepackage{amsmath,amscd}
\usepackage[mathscr]{euscript}
\usepackage[all]{xy}
\usepackage[utf8]{inputenc}
\usepackage{lmodern}
\usepackage[T1]{fontenc}
\usepackage[textwidth=15cm,hcentering]{geometry}
\usepackage[pagebackref=true,breaklinks=true,letterpaper=true,colorlinks]{hyperref}
\swapnumbers

\theoremstyle{plain}
\newtheorem{theo}{Theorem}[section]
\newtheorem{lemm}[theo]{Lemma}
\newtheorem{prop}[theo]{Proposition}
\newtheorem{coro}[theo]{Corollary}
\newtheorem*{theo*}{Theorem}

\theoremstyle{definition}
\newtheorem{defi}[theo]{Definition}

\newtheorem{cons}[theo]{Construction}

\theoremstyle{remark}
\newtheorem{rem}[theo]{Remark}
\newtheorem{exa}[theo]{Example}

\numberwithin{equation}{section}
\newcommand{\op}{^{\mathrm{op}}}
\newcommand{\cat}{\mathbf}
\newcommand{\oper}{\mathscr}
\newcommand{\on}{\mathrm}
\newcommand{\Emb}{\on{Emb}}
\newcommand{\Map}{\on{Map}}
\newcommand{\Mod}{\cat{Mod}}

\newcommand{\oMod}{\oper{M}od}

\newcommand{\oCom}{\oper{C}om}

\newcommand{\Man}{\cat{Man}}
\newcommand{\un}{\mathbb{I}}
\newcommand{\id}{\mathrm{id}}
\renewcommand{\S}{\cat{S}}

\newcommand{\goto}[1]{\stackrel{#1}{\longrightarrow}}
\newcommand{\HC}[1]{\mathrm{HH}_{#1}}
\newcommand{\HH}[1]{\mathrm{HH}^{#1}}
\newcommand{\Hom}{\underline{\mathrm{Hom}}}
\renewcommand{\hom}{\underline{\mathrm{hom}}}
\renewcommand{\L}{\mathbb{L}}
\newcommand{\R}{\mathbb{R}}
\newcommand{\Disk}{\cat{Disk}}

\title{Higher Hochschild cohomology of the Lubin-Tate ring spectrum}
\author{Geoffroy Horel}
\address{Mathematisches Institut\\
Einsteinstrasse 62\\
D-48149 Münster\\
Deutschland}
\email{geoffroy.horel@gmail.com}\thanks{The author was partially supported by an NSF grant.}

\begin{document}

\maketitle

\begin{abstract}
We give a method for computing factorization homology of an $\oper{E}_d$-algebra using as an input an algebraic version of higher Hochschild homology due to Pirashvili. We then show how to compute higher Hochschild homology and cohomology when the algebra is \'etale. As an application, we compute higher Hochschild cohomology of the Lubin-Tate ring spectrum.
\end{abstract}

\tableofcontents

This paper is devoted to higher Hochschild cohomology. Let us recall what this construction is. Given $E$ be an $\oper{E}_\infty$-ring spectrum. Hochschild cohomology of an associative algebra $A$ in $\Mod_E$ with coefficients in a bimodule $M$ is the derived homomorphisms object in the category of $A$-$A$-bimodules with source $A$ and target $M$. Higher Hochschild cohomology is the generalization of this construction when $A$ is an $\oper{E}_d$-algebra instead of an associative algebra. One needs to replace the notion of bimodule by the notion of operadic $\oper{E}_d$-module and the definition becomes
\[\HC{\oper{E}_d}(A|E,M)=\R\Hom_{\Mod_A^{\oper{E}_d}}(A,M)\]
where $\Hom_{\Mod_A^{\oper{E}_d}}$ denotes the homomorphism object in the category of operadic $\oper{E}_d$-modules over $A$. 

For practical reasons, we will use a different but equivalent definition of higher Hochschild cohomology inspired by factorization homology. For $A$ an $\oper{E}_d$-algebra in $\Mod_E$ and $V$ a $d$-dimensional framed manifold, there is a spectrum $\int_VA$ which is called the factorization homology of $A$ over $V$. This construction is functorial with respect to maps of $\oper{E}_d$-algebras and with respect to embeddings of framed $d$-manifolds. Moreover, $V\mapsto \int_V A$ is a symmetric monoidal functor. 

This easily implies that $\int_{S^{d-1}\times\mathbb{R}}A$ is an $\oper{E}_1$-algebra in spectra where $S^{d-1}\times\mathbb{R}$ is given a suitable framing. One can show that this $\oper{E}_1$-algebra serves as a universal enveloping algebra for the category of operadic $\oper{E}_d$-modules over $A$. More precisely, we prove in \ref{equivalence of the two definitions} the identity
\[\HC{\oper{E}_d}(A|E,M)\simeq \R\Hom_{A}^{S^{d-1}\times[0,1]}(A,M)\]
where the right hand side is an explicit construction given by a homotopy limit of a certain functor over the poset of disks on the manifold $S^{d-1}\times[0,1]$. In \ref{equivalence}, we prove an equivalence
\[\R\Hom_A^{S^{d-1}\times[0,1]}(A,M)\simeq \R\Hom^{[0,1]}_{\int_{S^{d-1}\times(0,1)}A}(A,M)\]
where the right hand side is a suitable generalization of the homomorphisms between left modules over an $\oper{E}_1$ (as opposed to associative) algebra.  Thus, we reduce the computation of Higher Hochschild cohomology to the computation of the derived homomorphisms between two left modules over an $\oper{E}_1$-algebra.

With this last description, we see that, in order to make explicit computations of higher Hochschild cohomology, the first step is to compute $\int_{S^{d-1}\times\mathbb{R}}A$ with its $\oper{E}_1$-structure. In the fifth section of this paper, we construct a spectral sequence that computes the factorization homology of an $\oper{E}_d$-algebra over a manifold:

\begin{theo*}{(\ref{spectral sequence})}
Let $A$ be an $\oper{E}_d$-algebra in $\Mod_E$, let $M$ be a framed $d$-manifold and $K$ be a homology theory with a $\mathbb{Z}/2$-equivariant K\"unneth isomorphism. There is a spectral sequence
\[\on{E}_{s,t}^2=\HH{M}_{s,t}(K_*A)\implies K_{s+t}(\int_MA)\]
\end{theo*}

Let us say a few words about the $\on{E}^2$-page. Given a commutative ring $k$, Pirashvili defines a functor $(X,A)\mapsto \HH{X}(A)$ where $X$ is a simplicial set, $A$ is a commutative algebra in $k$-modules and $\HH{X}(A)$ is a chain complex of $k$-modules. When $X=S^1$, this object is quasi-isomorphic to ordinary Hochschild homology. Our spectral sequence computing factorization homology is given by Pirashvili's higher Hochschild homology on the $\mathrm{E}^2$-page.

\medskip

In the sixth section, we make an explicit computation in the case of the Lubin-Tate spectrum (also known as Morava $E$-theory) $E_n$. Using the étaleness of the algebra $(K_n)_*E_n$ and the fact that $E_n$ is $K_n$-local, we can prove that for any $\oper{E}_d$-structure on $E_n$ that induces the correct multiplication on $K_n$-homology, the map
\[E_n\to \int_{S^{d-1}\times\mathbb{R}} E_n\]
is a $K_n$-homology equivalence. Using the fact that $E_n$ is $K_n$-local, we can prove

\begin{theo*}{(\ref{E_n hochschild cohomology})}
The map:
\[\HC{\oper{E}_d}(E_n)\to E_n\]
is a weak equivalence.
\end{theo*}

In the seventh section, we prove an \'etale base change theorem for \'etale algebras.

\begin{theo*}{(\ref{etale base change for Hochschild cohomology})}
Let $T$ be a commutative algebra in $\Mod_E$ that is ($K$-locally)  \'etale as an $\oper{E}_d$-algebras (more precisely, the $\oper{E}_d$-version of the cotangent complex of $E$ defined in \cite[Definition 2.7.]{francistangent} is ($K$-locally) contractible), then for any ($K$-local) $\oper{E}_d$-algebra $A$ over $T$, the base-change map 
\[\HC{\oper{E}_d}(A|E)\stackrel{\simeq}{\rightarrow}\HC{\oper{E}_d}(A|T)\]
is an equivalence.
\end{theo*}

In particular, this result combined with our computation implies that for any $K_n$-local $\oper{E}_d$-algebra $A$ over $E_n$, the map
\[\HC{\oper{E}_d}(A|E_n)\to\HC{\oper{E}_d}(A|\mathbb{S})\]
is a weak equivalence.

The full strength of the results proved in this paper is unnecessary in the case of $E_n$ since it is known to be a commutative ring spectrum. However, we think that the method presented here could be used in other contexts where one deals with $\oper{E}_d$-algebras that are not commutative.

\subsection*{Related work}

A geometric definition of higher Hochschild cohomology for commutative algebras is studied in \cite{ginothigher} and \cite{ginothigher2}. Our construction is a generalization to the  case of $\oper{E}_d$-algebras.

\subsection*{Acknowledgements}

This paper is part of the author's Ph.D. thesis. I wish to thank Haynes Miller, Clark Barwick, David Ayala, John Francis and Luis Alexandre Pereira for helpful conversations about the material of this paper.

\section*{Conventions}

We denote by $\S$ the category of simplicial sets with its usual model structure. We use boldface letters to denote categories. We use calligraphic letters like $\oper{A}$ to denote operads. All our categories and operads are enriched in $\S$. Note that given a topological operad or category, we can turn it into a simplically enriched operad or category by applying the functor $\on{Sing}$ to each mapping space. We allow ourselves to do this operation implicitly. 

We denote by $\Mod_E$ the simplicial category of modules over a commutative symmetric ring spectrum $E$. This category is symmetric monoidal for the relative tensor product over $E$. Moreover, it has two model structure : the positive model structure denoted $\Mod_E^+$ and the absolute model structure denoted $\Mod_E$. We refer the reader to the first section for more details. We often write $\cat{C}$ instead of $\Mod_E$ in the sections where the results do not depend a lot on the symmetric monoidal model category.

\section{Recollection on operads and factorization homology}

We recall a few notations. We denote by $\cat{Fin}$ the category whose objects are the nonnegative integers and with 
\[\cat{Fin}(m,n)=\cat{Set}(\{1,\ldots,m\},\{1,\ldots,n\})\]
We abuse notation and write $n$ for the finite set $\{1,\ldots,n\}$.

To an operad $\oper{M}$ with one color, we can assign its PROP $\cat{M}$. This is a category whose set of objects coincides with the set of objects of $\cat{Fin}$ and with
\[\cat{M}(m,n)=\bigsqcup_{f\in \cat{Fin}(m,n)}\prod_{i\in n}\oper{M}(f^{-1}(i))\]

Note that $\cat{Fin}$ is the PROP associated to the commutative operad. The construction of the associated PROP is a functor from operads to categories. In particular, the unique map $\oper{M}\to\oper{C}om$ induces a map $\cat{M}\to\cat{Fin}$.

An $\oper{M}$-algebra $A$ in a simplicially enriched symmetric monoidal category $\cat{C}$ induces a symmetric monoidal simplicial functor $\cat{M}\to\cat{C}$ that we will also denote by $A$.

Let $E$ be a commutative ring in symmetric spectra. We denote by $\Mod^+_E$ the category of modules over $E$ equipped with the positive model structure (constructed in \cite[Theorem III.3.2.]{schwedeuntitled} under the name projective positive stable model structure). The category $\Mod^+_E$ is a closed symmetric monoidal model category for the smash product over $E$ (denoted $-\otimes_E-$). It is also a simplicial model category. Moreover, the two structures are compatible in the sense that the tensor of simplicial sets and $E$-modules:
\[-\otimes-:\S\times\Mod^+_E\to\Mod^+_E\]
sending $(X,M)$ to $(E\wedge\Sigma_+^{\infty}X)\otimes_EM$ is a Quillen left bifunctor.

There is another model structure on $\Mod_E$ called the absolute model structure and that will be just denoted $\Mod_E$ (its construction can also be found in \cite[Thorem III.3.2.]{schwedeuntitled}). Its weak equivalences are the same as in the positive model structure but there are more cofibrations. In particular, the important fact for us is that the unit $E$ is cofibrant in the absolute model structure but not in the positive model structure. The model category $\Mod_E$ is also a closed symmetric monoidal simplicial model category. The advantage of the positive model structure is that the smash product is much better behaved. In particular, one can prove the following theorem which would be false for the absolute model structure.

\begin{theo}
The category $\Mod^+_E$ is a closed symmetric monoidal cofibrantly generated simplicial model category satisfying the following properties
\begin{itemize}
\item For any  operad $\oper{M}$ in $\S$, the category $\Mod^+_E[\oper{M}]$ of $\oper{M}$-algebras in $\Mod^+_E$ has a model category structure where weak equivalences and fibrations are created by the forgetful functors $\Mod^+_E[\oper{M}]\to(\Mod^+_E)^{\on{Col}(\oper{M})}$.

\item If $\alpha:\oper{M}\to\oper{N}$ is a is a map of operads, the adjunction
\[\alpha_!:\Mod^+_E[\oper{M}]\leftrightarrows\Mod^+_E[\oper{N}]:\alpha^*\]
is a Quillen adjunction which is a Quillen equivalence if $\alpha$ is a weak equivalence.

\item The forgetful functor $\Mod^+_E[\oper{M}]\to(\Mod_E)^{\on{Col}(\oper{M})}$ sends cofibrant objects to cofibrant objects.
\end{itemize}
\end{theo}

\begin{proof}
See \cite[Theorem 3.4.1. and 3.4.3.]{pavlovsymmetric}.
\end{proof}

\begin{rem}
All the operads, we consider in this work have a finite number of colors. The only kind of weak equivalences we will have to consider are maps that induce a bijection on the set of colors and induce weak equivalences on each space of operations.
\end{rem}

\subsection*{The little disk operad}

There is a topological category whose objects are $d$-manifolds without boundary and whose space of maps between $M$ and $N$ is $\Emb(M,N)$, the topological space of smooth embeddings with the weak $C^1$ topology. 

\begin{defi}
A \emph{framed $d$-manifold} is a pair $(M,\sigma_M)$ where $M$ is a $d$-manifold and $\sigma_M$ is a smooth section of the $\on{GL}(d)$-principal bundle $\on{Fr}(TM)$.
\end{defi}

If $M$ and $N$ are two framed $d$-manifolds, we define a space of framed embeddings denoted by $\Emb_f(M,N)$ as in \cite[Definition V.8.3.]{andrademanifolds}. We now recall this construction. First, given a diagram:
\[
\xymatrix{
 &Y\ar[d]^v\\
X\ar[r]_u& Z
}
\]
in the category of topological spaces over a fixed topological space $W$, we define its homotopy pullback as in \cite[V.9.]{andrademanifolds} to be the space of triples $(y,p,z)\in X\times Z^{[0,1]}\times Y$ such that $p(0)=u(x)$, $p(1)=v(y)$ and such that the image of $p$ in $W^{[0,1]}$ is a constant path. One can show that this is indeed a model for the homotopy pullback in the model category $\cat{Top}_{/W}$.

\begin{defi}\label{framed embedding}
Let $M$ and $N$ be two framed $d$-dimensional manifolds. The \emph{topological space of framed embeddings from $M$ to $N$}, denoted $\Emb_f(M,N)$, is given by the following homotopy pullback in the category of topological spaces over $\Map(M,N)$:
\[\xymatrix
{\Emb_f(M,N)\ar[r]\ar[d] & \Map(M,N)\ar[d]\\
\Emb(M,N)\ar[r] & \Map_{\on{GL}(d)}(\on{Fr}(TM),\on{Fr}(TN))}
\]

The right hand side map is obtained as the composition
\[\Map(M,N)\to\Map_{\on{GL}(d)}(M\times\on{GL}(d),N\times\on{GL}(d))\cong\Map_{\on{GL}(d)}(\on{Fr}(TM),\on{Fr}(TN))\]  
where the first map is obtained by taking the product with $\on{GL}(d)$ and the second map is induced by the identification $\on{Fr}(TM)\cong M\times\on{GL}(d)$ and $\on{Fr}(TN)\cong N\times \on{GL}(d)$ induced by our choice of framing on $M$ and $N$.
\end{defi}

Andrade explains in \cite[Definition V.10.1]{andrademanifolds} that there are well defined composition maps
\[\Emb_f(M,N)\times\Emb_f(N,P)\to\Emb_f(M,P)\]
allowing the construction of a topological category $f\Man_d$.

We denote by $D$ the open disk of dimension $d$.

\begin{prop}\label{framed embeddings of disks}
The evaluation at the center of the disks induces a weak equivalence
\[\Emb_f(D^{\sqcup p},M)\to\on{Conf}(p,M)\]
\end{prop}

\begin{proof}
See \cite[Proposition V.4.5.]{andrademanifolds} or \cite[Proposition 6.6.]{horelfactorization}.
\end{proof}

\begin{defi}
The \emph{little $d$-disks operad} $\oper{E}_d$ is the one-color operad whose $n$-th space is 
\[\oper{E}_d(n)=\Emb_f(D^{\sqcup n},D)\]
and whose composition is induced by composition of embeddings. We denote by $\cat{E}_d$ the PROP of the operad $\oper{E}_d$.
\end{defi}

\begin{rem}
This model of the little $d$-disk operad is introduced by Andrade in \cite{andrademanifolds}. Using \ref{framed embeddings of disks}, it is not hard to show that this definition is weakly equivalent to any other definition of the little $d$-disk operad.
\end{rem}

\subsection*{Factorization homology}

From now on and until we say otherwise, we denote by $(\cat{C}^+,\otimes,\un)$ the symmetric monoidal category $\Mod_E$ with its positive model structure and by $\cat{C}$ the same category equipped with the absolute model structure. We do this partly to simplify the notation bus mostly to emphasize that our arguments hold in greater generality modulo a few easy modifications.

\begin{defi}\label{factorization}
Let $A$ be a cofibrant object of $\cat{C}^+[\oper{E}_d]$. We define the \emph{factorization homology with coefficients in $A$} by the coend
\[\int_MA:=\Emb_f(-,M)\otimes_{\cat{E}_d}A\]
\end{defi}

This functor sends weak equivalences between cofibrant algebras to weak equivalences.

\begin{prop}
The functor $M\mapsto \int_MA$ is a simplicial and symmetric monoidal functor from the category $f\Man_d$ to the category $\cat{C}$.
\end{prop}

\begin{proof}
See \cite[Definition 7.3. and following paragraph]{horelfactorization}.
\end{proof}

Let $M$ be an object of $ f\Man_d$. Let $\cat{D}(M)$ be the poset of subsets of $M$ that are diffeomorphic to a disjoint union of disks. Let us choose for each object $V$ of $\cat{D}(M)$ a framed diffeomorphism $V\cong D^{\sqcup n}$ for some uniquely determined $n$. Each inclusion $V\subset V'$ in $\cat{D}(M)$ induces a morphism $D^{\sqcup n}\to D^{\sqcup n'}$ in $\cat{E}_d$ by composing with the chosen parametrization. Therefore each choice of parametrization induces a functor $\cat{D}(M)\to\cat{E}_d$. Up to homotopy this choice is unique since the space of automorphisms of $D$ in $ \cat{E}_d$ is contractible. 

In the following we assume that we have one of these functors $\delta:\cat{D}(M)\to\cat{E}_d$. We fix a cofibrant algebra $A: \cat{E}_d\to\cat{C}$.

\begin{prop}\label{left}
We have:
\[\int_M A\simeq\on{hocolim}_{V\in\cat{D}(M)}A(\delta V)\]
\end{prop}

\begin{proof}
See \cite[Corollary 7.7.]{horelfactorization}
\end{proof}

\section{Modules over $\oper{E}_d$-algebras}

We define the notion of an $S_\tau$-shaped module. These are modules over $\oper{E}_d$-algebras that are studied in details in \cite{horeloperads}.

\begin{defi}
A \emph{$d$-framing} of a closed $(d-1)$-manifold $S$ is a trivialization of the $d$-dimensional bundle $TS\oplus\mathbb{R}$ where $\mathbb{R}$ is a trivial line bundle. 
\end{defi}

For $M$ a $d$-manifold with boundary and $m$ a point of $\partial M$, we say that a vector $u\in T_mM$ is pointing inward if it is not in $T_m\partial M$ and if there is a curve $\gamma:[0,1)\to M$ whose derivative at $0$ is $u$.

\begin{defi}
Let $S$ be a closed $(d-1)$-manifold. An \emph{$S$-manifold} is a $d$-manifold with boundary $M$ together with the data of 
\begin{itemize}
\item a diffeomorphism $f:S\to\partial M$.
\item a non-vanishing section $\phi$ of the restriction of the vector bundle $TM$ on $\partial M$ which is such that $\phi(m)$ is pointing inward for any $m$ in $\partial M$.
\end{itemize}
\end{defi}

\begin{defi}\label{Stau manifolds}
Let $\tau$ be a $d$-framing of $S$. A \emph{framed $S_\tau$-manifold} is an $S$-manifold $(M,f,\phi)$ with the data of a framing of $TM$ such that the following composition sends $\tau$ to the given framing on the right-hand side.
\[TS\oplus\mathbb{R}\stackrel{Tf\oplus\mathbb{R}}{\longrightarrow}T(\partial M)\oplus\mathbb{R}\stackrel{i\oplus\phi}{\longrightarrow}TM_{|\partial M}\]
where the map $i$ is the obvious inclusion $T\partial M\to TM_{|\partial M}$.
\end{defi}

For $E\to M$ a $d$-dimensional vector bundle, we denote by $\on{Fr}(E)$ the $\on{GL}(d)$ bundle over $M$ whose fiber over $m$ is the space of basis of the vector space $E_m$. Note that a trivialization of $E$ is exactly the data of a section of $\on{Fr}(E)$.

For $(M,f,\phi)$ and $(M,g,\psi)$ two framed $S_\tau$-manifolds, we denote by $\Map^{S_\tau}_{\on{GL}(d)}(\on{Fr}(TM),\on{Fr}(TN))$ the space of morphisms of $\on{GL}(d)$-bundles whose underlying map $M\to N$ sends the boundary to the boundary and whose restriction to the boundary is fiberwise the identity (via the identification of both boundaries with $S$ and of both tangent bundles with $TS\oplus\mathbb{R}$).

\begin{defi}
Let $(M,f,\phi)$ and $(M,g,\psi)$ be two framed $S_\tau$-manifolds. The \emph{topological space of framed embeddings from $M$ to $N$}, denoted $\Emb_f^{S_\tau}(M,N)$, is the following homotopy pullback taken in the category of topological spaces over $\Map^S(M,N)$:
\[\xymatrix
{\Emb_f^{S_\tau}(M,N)\ar[r]\ar[d] & \Map^S(M,N)\ar[d]\\
\Emb^S(M,N)\ar[r] & \Map^{S_\tau}_{\on{GL}(d)}(\on{Fr}(TM),\on{Fr}(TN))}
\]
Any time we use the $S$ superscript, we mean that we are considering the subspace of maps commuting with the given map from $S$. 
\end{defi}

Recall that a right module over an operad $\oper{M}$ is an $\S$-enriched functor $\cat{M}\op\to\cat{S}$. We denote by $\Mod_{\oper{M}}$ the category of right modules over $\oper{M}$.

\begin{defi}\label{defi-Stau rigth modules}
Let $(S,\tau)$ be a $d$-framed  $(d-1)$-manifold. We define a right $\oper{E}_d$-module $S_\tau$ by the formula
\[S_\tau(n)=\Emb_f^{S_\tau}(D^{\sqcup n}\sqcup (S\times[0,1)),S\times[0,1))\]
\end{defi}

Recall, that there is a symmetric monoidal structure on $\Mod_{\oper{E}_d}$. If $F$ and $G$ are two objects of $\Mod_{\oper{E}_d}$, then their tensor product is the left Kan extension of the functor
\[(n,m)\mapsto F(n)\times G(m)\]
along the functor $\cat{Fin}\times\cat{Fin}\to\cat{Fin}$ sending a pair of finite sets to their disjoint union.

\begin{cons}
We give $S_\tau$ the structure of an associative algebra in $\Mod_{\oper{E}_d}$

Let $\phi$ be an element of $S_\tau(m)$ and $\psi$ be an element of $S_\tau(n)$. Let $\psi^S$ be the restriction of $\psi$ to $S\times[0,1)$. We define $\psi\square \phi$ to be the element of $S_\tau(m+n)$ whose restriction to $S\times[0,1)\sqcup D^{\sqcup m}$ is $\psi^S\circ\phi$ and whose restriction to $D^{\sqcup n}$ is $\psi_{|D^{\sqcup n}}$. 
\end{cons}

The operation $\square$
\[-\square-:S_\tau(n)\times S_\tau(m)\to S_\tau(n+m)\]
makes $S_\tau$ into an associative algebra in the symmetric monoidal category of right $\oper{E}_d$-modules.

\begin{defi}\label{def of Stau modules}
The colored operad $S_\tau\oMod$ has two colors $a$ and $m$. Its only non-empty spaces of operations are
\begin{align*}
S_\tau\oMod(a,\ldots,a;a)&=\oper{E}_d(n)\\
S_\tau\oMod(a,\ldots,a,m;m)&=S_\tau(n)
\end{align*}
where the $n$ on the right hand side is the number of $a$'s before the semicolon. The composition involves the operad structure on $\oper{E}_d$, the right $\oper{E}_d$-module structure on $S_\tau$ and the associative algebra structure on $S_\tau$.
\end{defi}

Again $(\cat{C}^+,\otimes,\un)$ denotes the symmetric monoidal model category $\Mod^+_E$ and $\cat{C}$ denotes the same category but with its absolute model structure. An algebra in $\cat{C}^+$ over $S_\tau\oMod$ consists of a pair of objects $(A,M)$ where $A$ is an $\oper{E}_d$-algebra and $M$ is equipped with an action of $A$ of the form
\[\Emb_f^{S_\tau}(S\times[0,1)\sqcup D^{\sqcup n},S\times[0,1))\otimes M\otimes A^{\otimes n}\to M\]

\begin{defi}
Let $A$ be an $\oper{E}_d$-algebra in $\cat{C}$. We define the \emph{category of $S_\tau$-shaped modules} over $A$ denoted $S_\tau\Mod_A$ to be the category whose objects are $S_\tau\oMod$-algebras whose restriction to the color $a$ is the $\oper{E}_d$-algebra $A$ and whose morphisms are morphisms of $S_\tau\oMod$-algebra inducing the identity map on $A$.
\end{defi}

\begin{rem}\label{general module}
More generally for any operad $\oper{O}$. The above construction gives a notion of modules over $\oper{O}$-algebras for third section of \cite{horeloperads}.
\end{rem}

\begin{prop}
Let $A$ be an $\oper{E}_d$-algebra in $\cat{C}$. The coend:
\[U_A^{S_\tau}=S_\tau\otimes_{\cat{E}_d}A\]
inherits an associative algebra structure from the one on $S_\tau$ and there is an equivalence of categories between the category of left modules over $U_A^{S_\tau}$ and the category $S_\tau\Mod_A$.
\end{prop}

\begin{proof}
See \cite[Proposition 3.9.]{horeloperads}.
\end{proof}

The previous proposition lets us put a model structure on $S_{\tau}\Mod_A$ in which the weak equivalences and fibrations are the maps that are sent to weak equivalences and fibrations by the forgetful functor $S_{\tau}\Mod_A\to\cat{C}$. Moreover, since $\cat{C}$ is a closed symmetric model category, the model category $S_{\tau}\Mod_A$ is a $\cat{C}$-enriched model category. 

\begin{exa}
The unit sphere inclusion $S^{d-1}\to\mathbb{R}^d$ has a trivial normal bundle. This induces a $d$-framing on $S^{d-1}$ which we denote $\kappa$. 
On the other hand we have the notion of an operadic module over an $\oper{E}_d$-algebra $A$. This is an object $M$ of $\cat{C}$ with multiplications maps
\[\oper{E}_d(n+1)\to\Map_{\cat{C}}(A^{\otimes n}\otimes M,M)\]
which are compatible with the $\oper{E}_d$-structure on $A$ in a suitable way (see \cite[Definition 1.1]{bergerderived}). We denote the category of such modules by $\Mod_A^{\oper{E}_d}$. The two notions are related by the following theorem.
\end{exa}

\begin{theo}\label{equivalence of modules}
For a cofibrant $\oper{E}_d$-algebra $A$, there is a Quillen equivalence
\[S_{\kappa}\Mod_A\leftrightarrows \Mod_A^{\oper{E}_d}\]
Moreover, the right adjoint of this equivalence commutes with the forgetful functor of both categories to $\cat{C}$.
\end{theo}

\begin{proof}
This is done in \cite[Proposition 4.12]{horeloperads}. The second claim follows from the fact that this equivalence is induced by a weak equivalence of associative algebra
\[U_A^{S^{d-1}_\kappa}\to U_A^{\oper{E}_d[1]}\]
where $U_A^{\oper{E}_d[1]}$ is the enveloping algebra of $\Mod_A^{\oper{E}_d}$ (i.e. it is an associative algebra such that there is an equivalence of categories $\Mod_{U_A^{\oper{E}_d[1]}}\simeq \Mod_A^{\oper{E}_d}$).
\end{proof}

Let $S$ be a closed $(d-1)$-manifold and let $\tau$ be a $d$-framing of $S$. There is a map $S_{\tau}\to\Emb_f(-,S\times(0,1))$ which sends an embedding $S\times[0,1)\sqcup D^{\sqcup n}\to S\times[0,1)$ to its restriction to $D^{\sqcup n}$. 

\begin{prop}
The map $S_\tau\to \Emb_f(-,S\times(0,1))$ is a weak equivalence of right $\oper{E}_d$-modules
\end{prop}

\begin{proof}
This follows from \cite[Proposition A.3.]{horeloperads}
\end{proof}

\begin{coro}\label{coro-equivalence enveloping factorization}
For a cofibrant $\oper{E}_d$-algebra $A$, there is a weak equivalence
\[U_A^{S_\tau}\goto{\simeq}\int_{S\times(0,1)}A\]
\end{coro}

\begin{proof}
By the previous proposition, there is a weak equivalence of right $\oper{E}_d$-modules
\[S_\tau\goto{\simeq} \Emb_f(-,S\times(0,1))\]
We prove in \cite[Proposition 2.8.]{horelfactorization} that for $A$ cofibrant, the functor $-\otimes_{\cat{E}_d}A$ preserves all weak equivalences of right $\oper{E}_d$-modules.
\end{proof}

Let $A$ be an $\oper{E}_d$-algebra, the factorization homology $\int_{S\times(0,1)}A$ is an $\oper{E}_1$ algebra. Indeed, any embedding $(0,1)^{\sqcup n}\to (0,1)$ induces an embedding $(0,1)\times S^{\sqcup n}\to (0,1)\times S$ by taking the product with $S$. Applying $\int_{-}A$ to this last embedding, we get maps
\[\Emb^f((0,1)^{\sqcup n},(0,1))\to\Map_{\cat{C}}\left((\int_{S\times(0,1)}A)^{\otimes n},\int_{S\times(0,1)}A\right)\]
We would like to say that the weak equivalence of the previous proposition is an equivalence of $\oper{E}_1$-algebra but it is not one on the nose. However, we show in the next proposition that this is a map of $S_\tau$-shaped modules.

\begin{prop}\label{equivalence of S shaped module}
There is an $S_\tau$-shaped module structure on $\int_{S\times(0,1)}A$ which is such that the map
\[U_A^{S_\tau}\to \int_{S\times(0,1)}A\]
is a weak equivalence of $S_\tau$-shaped modules.
\end{prop}

\begin{proof}
Let us describe the $S_\tau$-shaped module structure on $\int_{S\times(0,1)}A$. Let $\phi$ be a point in $\Emb^{S_\tau}_f(S\times[0,1)\sqcup D^{\sqcup n},S\times[0,1))$. By forgetting about the boundary, $\phi$ defines a point in $\Emb_f(S\times(0,1)\sqcup D^{\sqcup n},S\times(0,1))$ which induces a map
\[(\int_{S\times(0,1)}A)\otimes A^{\otimes n}\to\int_{S\times(0,1)}A\]

Letting $\phi$ vary, this gives $\int_{S\times(0,1)}A$ the structure of an $S_\tau$-shaped module. The map $U_A^{S_\tau}\to \int_{S\times(0,1)}A$ is then easily seen to be a map of $S_\tau$-shaped module. Since,we already know that it is a weak equivalence, we are done.
\end{proof}

\section{Higher Hochschild cohomology}

In this section, we construct a geometric model for higher Hochschild cohomology. We still denote by $(\cat{C},\otimes,\un)$ the symmetric monoidal model category $\Mod_E$. Our construction remains valid in other contexts (spaces, chain complexes, simplicial modules) modulo a few obvious modifications. We denote by $\Hom$ the inner Hom in the category $\cat{C}$. This functor is uniquely determined by the fact that we have a natural isomorphism
\[\cat{C}(X\otimes Y,Z)\cong \cat{C}(X,\Hom(Y,Z))\]

For any associative $R$ algebra in $\cat{C}$, the $\cat{C}$-enrichment of $\cat{C}$ induces to a $\cat{C}$-enrichment of $\Mod_R$. We denote by $\Hom_{\Mod_R}$ the hom-object in $\Mod_R$. 

Let $A$ be an $\oper{E}_d$-algebra which we assume to be cofibrant. Our goal is to construct a functor
\[\R\Hom^{S\times[0,1]}_A:S_\tau\Mod_A\op\times S_\tau\Mod_A\to\cat{C}\]
which is weakly equivalent to $\R\Hom_{S_\tau\Mod_A}(-,-)=\R\Hom_{\Mod_{U_A^{S_\tau}}}(-,-)$ but which is closer to the factorization homology philosophy.

For $S_{\tau}$ a $d$-framed $(d-1)$-manifold. We denote by $-\tau$, the $d$-framing on $S$ obtained by pulling back $\tau$ along the isomorphism of the vector bundle $TS\oplus\mathbb{R}$ which is the identity on the first summand and multiplication by $-1$ on the second summand.

In particular, $S\times[0,1)$ is naturally an $S_{\tau}$ manifold and $S\times (0,1]$ is an $S_{-\tau}$-manifold.

\begin{defi}\label{defi-disks with boundaries}
We denote by $\cat{Disk}_d^{S_\tau\sqcup S_{-\tau}}$ the topological category whose objects are the $S_\tau\sqcup S_{-\tau}$-manifolds of the form $S\times[0,1)\sqcup D^{\sqcup n}\sqcup S\times(0,1]$ with $n$ in $\mathbb{Z}_{\geq 0}$ and whose morphisms are given by the spaces $\Emb^{S_\tau\sqcup S_{-\tau}}_f$. 
\end{defi}

\begin{cons}\label{Loday functor}
We define a functor 
\[\mathscr{F}(M,A,N):(\cat{Disk}_d^{S_\tau\sqcup S_{-\tau}})\op\to \cat{C}\]

Its value on $S\times[0,1)\sqcup D^{\sqcup n}\sqcup S\times(-1,0]$ is $\Hom(M\otimes A^{\otimes n},N)$. 

Notice that any map in $(S_\tau\sqcup S_{-\tau})\Mod$ can be decomposed as a disjoint union of embeddings of the following three types:
\begin{itemize}
\item An embedding $S\times[0,1)\sqcup D^{\sqcup k}\to S\times[0,1)$.

\item An embedding $D^{\sqcup l}\to D$ (where $l$ is possibly zero).

\item An embedding $D^{\sqcup m}\sqcup S\times(0,1]\to S\times(0,1]$.
\end{itemize}

Let $\phi$ be an embedding $S\times[0,1)\sqcup D^{\sqcup n}\sqcup S\times(0,1]\to S\times[0,1)\sqcup D^{\sqcup m}\sqcup S\times(0,1]$ and let
\[\phi=\phi_{+}\sqcup \psi_1\sqcup \ldots\sqcup \psi_r\sqcup \phi_-\]
be its decomposition with $\phi_+$ of the first type, $\phi_-$ of the third type and $\psi_i$ of the second type for each $i$. We need to extract from this data a map
\[\Hom(M\otimes A^{\otimes m},N)\to \Hom(M\otimes A^{\otimes n},N)\]

The action of $\phi_+$ and of the $\psi_i$ are constructed in an obvious way from the $\oper{E}_d$-structure of $A$ and the $S_\tau$-shaped module structure on $M$. The only non trivial part is the action of $\phi_-$. We can hence assume that $\phi=\id_{S\times[0,1)\sqcup D^{\sqcup p}}\sqcup \phi_-$ where $\phi_-$ is an embedding $D^{\sqcup n}\sqcup S\times(0,1]\to S\times(0,1]$. We want to construct
\[\Hom(M\otimes A^{\otimes p},N)\to\Hom(M\otimes A^{\otimes p}\otimes A^{\otimes n},N)\]

First, observe that there is a diffeomorphism $S\times[0,1)\to S\times(0,1]$ sending $(s,t)$ to $(s,1-t)$. This diffeomorphism sends the framing $\tau$ on $S\times[0,1)$ to the framing $-\tau$ on $S\times(0,1]$. Similarly, reflexion about the hyperplane $x_1=0$ induces a diffeomorphism $D\to D$. (Recall that $D=\{(x_1,\ldots,x_d),\sum x_i^2<1\}$. Conjugating by this diffeomorphism, the embdding $\phi_-$ induces an embedding $\tilde{\phi}_+$
\[\tilde{\phi}_-:S\times[0,1)\sqcup D^{\sqcup n}\to S\times[0,1)\]

In fact, it is straightforward to see that this construction induces a homeomorphism
\[\Emb^{S_{-\tau}}_f(S\times(0,1]\sqcup D^{\sqcup n},S\times(0,1])\to\Emb^{S_\tau}_f(S\times[0,1)\sqcup D^{\sqcup n},S\times[0,1))\]

Now, notice that $\Hom(M\otimes A^{\otimes p},N)$ has the structure of an $S_\tau$-shaped $A$ module induced from the one on $N$. Thus, the map $\tilde{\phi}_-$ induces a map:
\[\Hom(M\otimes A^{\otimes p},N)\otimes A^{\otimes n}\to \Hom(M\otimes A^{\otimes p},N)\]

This map is adjoint to
\[\Hom(M\otimes A^{\otimes p},N) \to \Hom(M\otimes A^{\otimes p}\otimes A^{\otimes n},N)\]
which we define to be the action of $\phi$.
\end{cons}

\begin{rem}
In order to be homotopically meaningful, we need a derived version of $\mathscr{F}(M,A,N)$. We claim that the homotopy type of $\mathscr{F}(M,A,N)$ only depends on the homotopy type of $M$, $A$ and $N$ as long as $A$ is a cofibrant $\oper{E}_d$-algebra and $M$ is a cofibrant object of $S_\tau\Mod_A$ and $N$ is fibrant. Indeed, these conditions imply that 
\begin{itemize}
\item The object $M$ is cofibrant in $\cat{C}$. This is because the forgetful functor $S_\tau\Mod_A\to\cat{C}$ preserves cofibrations.
\item $A$ is cofibrant in $\cat{C}$.
\end{itemize} 

This implies that for all $k$, $\Hom(M\otimes A^{\otimes k},N)\simeq \R\Hom(M\otimes A^{\otimes k},N)$.

\end{rem}

Let  $\cat{A}$ be a small category, $F$ a functor from $\cat{A}$ to $\cat{S}$ and $G$ a functor from $\cat{A}$ to $\cat{C}$. We denote by $\hom_{\cat{A}}(F,G)$ the end
\[\int_{\cat{A}}\hom(F(-),G(-))\]

We denote by $\R\hom_{\cat{A}}(F,G)$ the derived functor which is obtained by taking a cofibrant replacement of the source and a fibrant replacement of the target.

\begin{defi}\label{Hom over cylinder}
We define $\R\Hom^{S\times[0,1]}_A(M,N)$ to be the homotopy end
\[\R\hom_{(\cat{Disk}_d^{S_\tau\sqcup S_{-\tau}})\op}(\Emb_f^{S_\tau\sqcup S_{-\tau}}(-,S\times[0,1]),\mathscr{F}(QM,A,RN))\]
where $QM\to M$ is a cofibrant replacement in $S_\tau\Mod_A$ and $N\to RN$ is a fibrant replacement.
\end{defi}

The main theorem of this section is the following

\begin{theo}\label{right hom}
There is a weak equivalence:
\[\R\Hom^{S\times[0,1]}_A(M,N)\simeq\R\Hom_{S_\tau\Mod_A}(M,N)\]
\end{theo}

The rest of this section is devoted to the proof of this theorem. The reader willing to admit this result can safely skip the proof and move directly to the last subsection of this section.

\subsection*{Case of $\oper{E}_1$-algebras}

The one-point space is $0$-manifold. This manifold has two $1$-framing that we call the negative and positive framing. By definition, a $1$-framing of the point is the data of a basis of $\mathbb{R}$ as a $\mathbb{R}$-vector space. The positive framing is the one given by $1$ and the negative framing is the one given by $-1$. Thus, by definition \ref{defi-Stau rigth modules}, we get two right modules over $\oper{E}_1$. We denote by $\mathscr{R}$ the one corresponding to the negative framing and $\mathscr{L}$ the one corresponding to the positive framing.

\begin{defi}
A \emph{left module} over an $\oper{E}_1$-algebra $A$ is an object of the category $\mathscr{L}\Mod_A$. Similarly, a \emph{right module} over $A$ is an object of $\mathscr{R}\Mod_A$.
\end{defi}

More explicitly, an object of $\mathscr{L}\Mod_A$ is an object of $\cat{C}$, $M$ together with multiplication maps
\[A^{\otimes n}\otimes M\to M\]
for each embedding $[0,1)\sqcup (0,1)^{\sqcup n}\to [0,1)$
These maps are moreover  supposed to satisfy a unitality and associativity condition.

We denote by $\cat{Disk}_1^{-+}$ the one dimensional version of the category $\cat{Disk}^{S_\tau\sqcup S_{-\tau}}$ defined in \ref{defi-disks with boundaries}. As a particular case of \ref{Hom over cylinder}, given a cofibrant $\oper{E}_1$-algebra $A$ and two left modules $M$ and $N$, we can define $\Hom^{[0,1]}_A(M,N)$ and this is given by natural transformations between contravariants functors on $\cat{Disk}_1^{-+}$.

\begin{defi}
The \emph{category of non-commutative intervals} denoted $\cat{Ass}^{-+}$ is a skeleton of the category whose objects are finite sets containing $\{-,+\}$ and whose morphisms are maps of finite sets $f$ preserving $-$ and $+$ together with the extra data of a linear ordering of each fiber which is such that $-$ (resp. $+$) is the smallest (resp. largest) element in the fiber over $-$ (resp $+$).
\end{defi}

Note that the functor $\pi_0$ which sends a disjoint union of intervals to the set of connected components is an equivalence of topological categories from  $\Disk_1^{-+}$ to $\cat{Ass}^{-+}$. In fact, we could have defined $\cat{Ass}^{-+}$ as the homotopy category of $\Disk_1^{-+}$.

Let $A$ be an associative algebra and $M$ and $N$ be left modules over it. We define $F(M,A,N)$ to be the obvious functor $(\cat{Ass}^{-+})\op\to\cat{C}$ sending $\{-,1,\ldots,n,+\}$ to $\Hom( A^{\otimes n}\otimes M,N)$. The functoriality is defined analogously to \ref{Loday functor}.

Recall that $\Delta\op$ can be described as a skeleton of the category whose objects are linearly ordered sets with at least two elements and morphisms are order preserving morphisms preserving the minimal and maximal element.

With this description, there is an obvious functor $\Delta\op\to\cat{Ass}^{-+}$ which sends a totally ordered set with minimal element $-$ and maximal element $+$ to the underlying finite set and an order preserving map to the underlying map with the data of the induced linear ordering of each fiber.

Recall that given a triple $(M,A,N)$ consisting of an associative algebra $A$ and two left modules $M$ and $N$, we can form the cobar construction $C^{\bullet}(M,A,N)$ which is a cosimplicial object of $\cat{C}$ whose value at $[n]$ is $\Hom(A^{\otimes n}\otimes M,N)$. It is classical that if $A$ and $M$ are cofibrant and $N$ is fibrant, then $C^{\bullet}(M,A,N)$ is Reedy fibrant and its totalization is a model for the derived Hom $\R\Hom_{\Mod_A}(M,N)$.

\begin{prop}
Let $A$ be an associative algebra and $M$ and $N$ be left modules over it. The composition of $F(M,A,N)$ with the functor $\Delta\to(\cat{Ass}^{-+})\op$ is the cobar construction $C^\bullet(M,A,N)$
\end{prop}

\begin{proof}
This is a straightforward computation.
\end{proof}

We denote by $P:(\cat{Ass}^{-+})\op\to\S$ the left Kan extension of the cosimplicial space which is levelwise a point along the map $\Delta\to(\cat{Ass}^{-+})\op$. Concretely $P$ sends a finite set with two distinguished elements $-$ and $+$ to the set of linear ordering of that set whose smallest element is $-$ and largest element is $+$ seen as a discrete space.

\begin{coro}
Let $A$ be a cofibrant associative algebra and $M$ and $N$ be left modules over it. Then
\[\R\Hom_A(M,N)\simeq \R\hom_{\cat{Ass}^{-+}}(P,F(M,A,N))\]
\end{coro}

\begin{proof}
Assume that $M$ is cofibrant and $N$ is fibrant. If they are not , we take an appropriate replacement. The left hand side is
\[\on{Tot}([n]\to C^n(M,A,N)=\Hom(M\otimes A^{\otimes n},N))\]

According to the cofibrancy/fibrancy assumption, this cosimplicial functor is Reedy fibrant, therefore the totalization coincides with the homotopy limit. Hence we have
\[\R\Hom_A(M,N)\simeq \R\hom_{\Delta}(*,C^\bullet(M,A,N))\simeq \R\hom_{\cat{Ass}^{-+}}(P,F(M,A,N))\]
\end{proof}

\begin{prop}
Let $A$ be a cofibrant associative algebra and $M$ and $N$ be left modules over it. Then there is a weak equivalence
\[\R\Hom_A^{[0,1]}(M,N)\goto{\simeq}\R\Hom_A(M,N)\]
\end{prop}

\begin{proof}
The right hand side is the derived end
\[\R\hom_{\cat{Ass}^{-+}}(P,F(M,A,N))\]
which can be computed as the totalization of the Reedy fibrant cosimplicial object
\[C^\bullet(P,\cat{Ass}^{-+},F(M,A,N))\]

Similarly, the left hand side is the totalization of the Reedy fibrant cosimplicial object
\[C^\bullet(\Emb^{-+}(-,[0,1]),\Disk^{-+},\mathscr{F}(M,A,N))\]

There is an obvious map of cosimplicial objects
\[C^\bullet(\Emb^{-+}(-,[0,1]),\Disk^{-+},\mathscr{F}(M,A,N))\to C^\bullet(P,\cat{Ass}^{-+},F(M,A,N))\]
which is degreewise a weak equivalence. Therefore, there is a weak equivalence between the totalizations
\[\R\Hom_A^{[0,1]}(M,N)\goto{\simeq}\R\Hom_A(M,N)\]
\end{proof}

If $A$ is an $\oper{E}_1$-algebra, it can be seen as an object of $\mathscr{L}\Mod_A$ as follows. The map
\[A\otimes A^{\otimes n}\to A\]
corresponding to an embedding
\[\phi:[0,1)\sqcup (0,1)^{\sqcup n}\to [0,1)\]
is defined to be the multiplication map $A^{\otimes n+1}\to A$ corresponding to the restriction of $\phi$ to its interior.

We denote by $(A,A^m)$ the $\mathscr{L}\oMod$-algebra consisting of $A$ acting on itself in the above way.

\begin{coro}\label{E_1 vs Ass}
Let $A$ be a cofibrant $\oper{E}_1$-algebra and $N$ a left module. Then 
\[\R\Hom_A^{[0,1]}(A^m,N)\simeq N\]
\end{coro}

\begin{proof}
The pair $(A,N)$ forms an algebra over $\mathscr{L}\oMod$. The operad $\mathscr{L}\oMod$ is weakly equivalent to the operad $L\oMod$ parameterizing strictly associative algebras and left modules. This implies that we can find a pair $(A',N')$ consisting of an associative algebra and a left module together with a weak equivalence of $\mathscr{L}\oMod$-algebra
\[(A,N)\goto{\simeq}(A',N')\]
Using the previous proposition, we have
\[\R\Hom_A^{[0,1]}(A^m,N)\simeq\R\Hom_{A'}(A',N')\simeq N'\simeq N\]
\end{proof}

Let $\cat{D}([0,1])$ be the poset of open sets of $[0,1]$ that are diffeomorphic to $[0,1)\sqcup (0,1)^{\sqcup n}\sqcup (0,1]$ for some $n$. Let us choose a functor
\[\delta:\cat{D}([0,1])\to \Disk^{-+}\]
by picking a parametrization of each object of $\cat{D}([0,1])$.

\begin{prop}\label{hom as a holim}
There is a weak equivalence
\[\R\Hom^{[0,1]}_A(M,N)\simeq \on{holim}_{U\in\cat{D}([0,1])\op}\mathscr{F}(M,A,N)(\delta U)\]
\end{prop}

\begin{proof}
We can assume that $M$ is cofibrant and $N$ is fibrant. First, we have the equivalence
\[\R\Hom^{[0,1]}_A(M,N)\simeq \on{holim}_{U\in\cat{D}([0,1])\op}\R\Hom_A^{\delta U}(M,N)\]
which follows easily from the following equivalence in the category of contravariant functors on $\Disk^{-+}$:
\[\Emb_f^{S^0}(-,[0,1])\simeq\on{hocolim}_{U\in\cat{D}([0,1])}\Emb_f^{S^0}(-,U)\]
which is proved in \cite[Lemma 7.9.]{horelfactorization}.

Then we notice, using Yoneda's lemma, that $U\mapsto\R\Hom^{\delta U}_A(M,N)$ is weakly equivalent as a functor to $U\mapsto \mathscr{F}(M,A,N)(\delta U)$.
\end{proof}

\subsection*{Comparison with the actual homomorphisms}

In this subsection, $A$ is a cofibrant $\oper{E}_d$-algebra. We want to compare $\R\Hom^{S\times[0,1]}_A(M,N)$ with $\R\Hom_{S_\tau\Mod_A}(M,N)$.

\begin{cons}
Let $M$ be an $S_\tau$-shaped module over an $\oper{E}_d$-algebra $A$. We give $M$ the structure of a left module over the $\oper{E}_1$-algebra $\int_{S\times(0,1)}A$. Let
\[ (0,1)^{\sqcup n}\sqcup[0,1)\to [0,1)\]
be a framed embedding. We can take the product with $S$ and get an embedding in $f\Man_d^{S_\tau}$
\[(S\times(0,1))^{\sqcup n}\sqcup S\times [0,1) \to S\times [0,1)\]
Evaluating $\int_{-}(M,A)$ over this embedding, we find a map
\[(\int_{S\times(0,1)}A)^{\otimes n}\otimes M\to M\]

All these maps give $M$ the structure of a left $\int_{S\times(0,1)}A$-module.
\end{cons}

\begin{prop}
Let $M$ and $N$ be two $S_{\tau}$-shaped modules over $A$. There is a weak equivalence
\[\R\Hom^{S\times[0,1]}_A(M,N)\simeq\on{holim}_{U\in\cat{D}([0,1])\op}\mathscr{F}(M,\int_{S\times(0,1)}A,N)(S\times U)\]
where $M$ and $N$ are given the structure of left $\int_{S\times(0,1)}A$-modules using the previous construction.
\end{prop}

\begin{proof}
This is a straightforward variant of \ref{hom as a holim}. One first proves that 
\[\R\Hom_A^{S\times[0,1]}(M,N)\simeq\on{holim}_{U\in\cat{D}([0,1])\op}\R\Hom_A^{S\times U}(M,N)\]
which follows from the following equivalence in the category $\on{Fun}((\Disk^{S_\tau\sqcup S_{-\tau}})\op,\S)$:
\[\on{hocolim}_{U\in\cat{D}([0,1])}\Emb_f^{S_\tau\sqcup S_{-\tau}}(-,S\times U)\simeq\Emb_f^{S_\tau\sqcup S_{-\tau}}(-,S\times [0,1])\]
and then, using Yoneda's lemma it is easy to check that the functor
\[U\mapsto\R\Hom_A^{S\times U}(M,N)\]
is weakly equivalent to
\[U\mapsto \mathscr{F}(M,\int_{S\times(0,1)}A,N)( U)\]
\end{proof}

\begin{coro}\label{equivalence}
There is a weak equivalence
\[\R\Hom_{\int_{S\times(0,1)}A}^{[0,1]}(M,N)\simeq \R\Hom_A^{S\times[0,1]}(M,N)\]
\end{coro}

\begin{proof}
Both sides are weakly equivalent to
\[\on{holim}_{U\in\cat{D}([0,1])\op}\mathscr{F}(M,\int_{S\times(0,1)}A,N)(S\times U)\]
One side by the previous proposition and the other by \ref{hom as a holim}.
\end{proof}

\subsection*{Proof of \ref{right hom}}

\begin{proof}
If we fix $A$ and a fibrant $S_\tau$-shaped module $N$ and let $M$ vary, we want to compare two functors from $S_{\tau}\Mod_A$ to $\cat{C}$. Both functors preserve weak equivalences between cofibrant objects and turn homotopy colimits into homotopy limits, therefore, it suffices to check that both functors are weakly equivalent on the generator of the category of $S_\tau$-shaped modules. In other word, it is enough to prove that
\[\R\Hom^{S\times[0,1]}_A(U_A^{S_\tau},N)\simeq\R\Hom_{S_\tau\Mod_A}(U_A^{S_\tau},N)\]

The right hand side of the above equation can be rewritten as $\R\Hom_{U_A^{S_\tau}}(U_A^{S_\tau},N)$ which is trivially weakly equivalent to $N$.

We know from \ref{equivalence of S shaped module} that as $S_\tau$-shaped modules, there is a weak equivalence
\[U_A^{S_\tau}\to\int_{S\times(0,1)}A\]

Therefore, it is enough to prove that there is a weak equivalence
\[\R\Hom^{S\times[0,1]}_A(\int_{S\times(0,1)}A,N)\simeq N\]

According to \ref{equivalence}, it is equivalent to prove that there is a weak equivalence:
\[\R\Hom^{[0,1]}_{\int_{S\times[0,1]}A}(\int_{S\times(0,1)}A,N)\simeq N\]
which follows directly from \ref{E_1 vs Ass}.
\end{proof}

\subsection*{A generalization}\label{subsection generalization}

We can generalize the definition \ref{Hom over cylinder}. In \cite[Construction 6.9.]{horeloperads}, given the data of a framed bordism $W$ between $d$-framed manifold of dimension $d-1$ $S_\sigma$ and $T_\tau$, we construct a left Quillen functor
\[P_W:S_{\sigma}\Mod_A\to T_{\tau}\Mod_A\]

The best way to think of this functor is as follows. Factorization homology of $A$ over $W$ is a $U_A^{S_\sigma}$-$U_A^{T_\tau}$-bimodule. Thus tensoring with it induces a left Quillen functor
\[S_\sigma\Mod_A\to T_\tau\Mod_A\]

\begin{cons}\label{Hom over a bordism}
Let $W$ be bordism from $S_\sigma$ to $T_\tau$. Let $M$ be an $S_\sigma$-shaped module over $A$ and $N$ be a $T_{\tau}$-shaped module. We can construct a functor as in \ref{Loday functor} $\mathscr{F}(M,A,N)$ from $(\Disk^{S_{\sigma}\sqcup T_{-\tau}})\op$ to $\cat{C}$  which sends $S\times[0,1)\sqcup D^{\sqcup n}\sqcup T\times(0,1]$ to $\Hom( A^{\otimes n}\otimes M,N)$. We define $\R\Hom^{W}_A(M,N)$ to be the homotopy end
\[\R\Hom^{W}_A(M,N)=\R\hom_{(\Disk^{S_{\sigma}\sqcup T_{-\tau}})\op}(\Emb^{S_\sigma\sqcup T_{-\tau}}_f(-,W),\mathscr{F}(M,A,N))\]
\end{cons}

This construction has the following nice interpretation:

\begin{theo}\label{what is Hom over a bordism}
Let $W$ be a bordism from $S_\sigma$ to $T_\tau$. There is a weak equivalence:
\[\R\Hom_A^{W}(M,N)\simeq \R\Hom_A^{T\times[0,1]}(\L P_W(M),N)\]
\end{theo}

\begin{proof}
The proof is very analogous to the proof of \ref{right hom}.
\end{proof}

We now introduce the definition of higher Hochschild cohomology.

\begin{defi}\label{defi-Higher Hochschild cohomology}
Let $A$ be a cofibrant $\oper{E}_d$-algebra in $\cat{C}$ and $M$ be a $S^{d-1}_\kappa$-shaped module over $A$. The \emph{$\oper{E}_d$-Hochschild cohomology} of $A$ with coefficients in $M$ is
\[\HC{\oper{E}_d}(A,M)=\R\Hom_{S^{d-1}_{\kappa}\Mod_A}(A,M)\]
\end{defi}

Now, we compare this definition to a more traditional definition. Let $A$ be a cofibrant $\oper{E}_d$-algebra and $M$ be an object of $\Mod_A^{\oper{E}_d}$. By \ref{equivalence of modules}, we can see $M$ as a $S^{d-1}_\kappa$-shaped module over $A$. 

\begin{prop}\label{equivalence of the two definitions}
For $A$ a cofibrant $\oper{E}_d$-algebra and $M$ an object of $\Mod_A^{\oper{E}_d}$, we have a weak equivalence
\[\R\Hom_{\Mod_A^{\oper{E}_d}}(A,M)\simeq \R\Hom_{S_{\kappa}\Mod_A}(A,M)\]
\end{prop}

\begin{proof}
By \ref{equivalence of modules}, we have a Quillen equivalence
\[u_!:S^{d-1}_\kappa\Mod_A\leftrightarrows \Mod_A^{\oper{E}_d}:u^*\]
Therefore, we have a weak equivalence $\L u_!u^*A\to A$ in $\Mod_A^{\oper{E}_d}$. This gives us a weak equivalence
\[\R\Hom_{\Mod_A^{\oper{E}_d}}(A,M)\to\R\Hom_{\Mod_A^{\oper{E}_d}}(\L u_!u^*A,M)\simeq \R\Hom_{S_{\kappa}\Mod_A}(u^*A,u^*M)\]
\end{proof}

Thus, our definition of $\HC{\oper{E}_d}(A,M)$ coincides with the more traditional definition we gave in the first paragraph of the introduction. According to \ref{right hom}, we have a weak equivalence $\HC{\oper{E}_d}(A,M)\simeq\R\Hom_{A}^{S^{d-1}\times[0,1]}(A,M)$. As usual, we write $\HC{\oper{E}_d}(A)$ for $\HC{\oper{E}_d}(A,A)$. 

\begin{prop}
Let $\bar{D}$ be the closed unit ball in $\mathbb{R}^d$ seen as a bordism from the empty manifold to $S^{d-1}_\kappa$. There is a weak equivalence:
\[\HC{\oper{E}_d}(A,M)\simeq \R\Hom^{\bar{D}}_A(\un,M)\]
\end{prop}

\begin{proof}
$\un$, the unit of $\cat{C}$ is an object of $\varnothing\Mod_A$ (note that $\varnothing\Mod_A$ is equivalent to the category $\cat{C}$) and $\L P_{\bar{D}}(\un)$ is weakly equivalent to $A$. Then it suffices to apply \ref{what is Hom over a bordism}.
\end{proof}

This has the following surprising consequence.

\begin{coro}
The group $\on{Diff}_f^{S^{d-1}}(\bar{D})$ acts on $\HC{\oper{E}_d}(A,M)$.
\end{coro}

\begin{rem}
The group $\on{Diff}_f^{S^{d-1}}(\bar{D})$ is weakly equivalent to the homotopy fiber of the inclusion
\[\on{Diff}^{S^{d-1}}(\bar{D})\to\on{Imm}^{S^{d-1}}(\bar{D},\bar{D})\]
where the $S^{d-1}$ superscript means that we are restricting to the diffeomorphisms or immersions which are the identity outside on $S^{d-1}=\partial\bar{D}$. In fact the action of $\on{Diff}_f^{S^{d-1}}(\bar{D})$ factors through the inverse limit of the embedding calculus tower computing this group. Since we are in the codimension $0$ case, the embedding calculus tower should not be expected to converge. Even if it does not converge, it is an interesting mathematical object. In particular, using the work of Arone and Turchin in \cite{aronerational} and Willwacher in  \cite[Theorem 1.2.]{willwacherkontsevich}, we get an action of the Grothendieck-Teichm\"uller Lie algebra $\mathfrak{grt}$ on the $\oper{E}_2$-Hochschild cohomology of an algebra over $H\mathbb{Q}$. We hope to further study this action in future work.
\end{rem}

\section{Pirashvili's higher Hochschild homology}

Let $R$ be a commutative graded ring.  We denote by $\cat{Ch}_{\geq 0}(R)$ the category of non-negatively graded chain complexes. This has a model category structure in which the weak equivalences are the quasi-isomorphisms, the cofibrations are the degreewise monomorphisms with degreewise projective cokernel and the fibrations are the epimorphisms. In particular, any object is fibrant and the cofibrant objects are the degreewise projective chain complexes.

The model category $\cat{Ch}_{\geq 0}(R)$ is cofibrantly generated. Thus, we have the projective model category structure on functors $\cat{Fin}\to\cat{Ch}_{\geq 0}(R)$, in which weak equivalences and fibrations are objectwise. The following definition is due to Pirashvili (see \cite[Introduction, p.151]{pirashvilihodge}, see also \cite[Definition 2.]{ginotderived}).

\begin{defi}
Let $A$ be a degreewise projective commutative algebra in $\cat{Ch}_{\geq 0}(R)$ and let $X$ be a simplicial set. We denote by $\on{HH}^X(A|R)$ the homotopy coend
\[\Map(-,X)\otimes^{\L}_{\cat{Fin}}A\]
\end{defi}

\begin{rem}
In practice, we can take $\on{HH}^X(A|R)$  to be the realization of the simplicial object
\[\on{B}_{\bullet}(\Map(-,X),\cat{Fin},A)\]
This construction preserves quasi-isomorphism between degreewise projective commutative algebras. In the following $\on{HH}^X(A|R)$ will be taken to be this explicit model.

This construction also sends a weak equivalence $X\goto{\simeq} Y$ to a weak equivalence \[\on{HH}^X(A|R)\goto{\simeq}\on{HH}^Y(A|R)\]
\end{rem}

\begin{prop}
Let $A$ be a degreewise projective commutative algebra in $\cat{Ch}_{\geq 0}(R)$, then the functor $X\mapsto\HH{X}(A|R)$ lifts to a functor from $\S$ to the category of commutative algebra in $\cat{Ch}_{\geq 0}(R)$.
\end{prop}

\begin{proof}
The category $\on{Fun}(\cat{Fin}\op,\S)$ equipped with the convolution tensor product is  a symmetric monoidal model category (see \cite[Proposition 2.2.15]{isaacsoncubical}). It is easy to check that there is an isomorphism:
\[\Map(-,X)\otimes\Map(-,Y)\cong\Map(-,X\sqcup Y)\]
Moreover, since $A:\cat{Fin}\to\cat{Ch}_{\geq 0}(R)$ is a commutative algebra for the convolution tensor product, the objects $\HH{X}(A|R)$ is a symmetric monoidal functor in the $X$ variable. To conclude, it suffices to observe that any simplicial set is a commutative monoid with respect to the disjoint union in a unique way and that this structure is preserved by maps in $\S$. Therefore, $\HH{X}(A|R)$ is a commutative algebra functorially in $X$.
\end{proof}

\begin{prop}\label{HH pushout}
Let $A$ be a degreewise projective commutative algebra in $\cat{Ch}_{\geq 0}(R)$. Let
\[\xymatrix{
X\ar[r]\ar[d]&Z\ar[d]\\
Y\ar[r]&P
}\]
be a homotopy pushout in the category of simplicial sets. Then there is a weak equivalence
\[\HH{P}(A|R)\simeq |\on{B}_\bullet(\HH{Y}(A|R),\HH{X}(A|R),\HH{Z}(A|R))|\]
\end{prop}

\begin{proof}
First, notice that the maps $X\to Z$ and $X\to Y$ induce commutative algebra maps $\HH{X}(A|R)\to\HH{Y}(A|R)$ and $\HH{X}(A|R)\to\HH{Z}(A|R)$. In particular $\HH{Z}(A|R)$ and $\HH{Y}(A|R)$ are modules over $\HH{X}(A|R)$. This explains the bar construction in the statement of the proposition.

We can explicitly construct $P$ as the realization of the following simplicial space
\[[p]\mapsto Y\sqcup X^{\sqcup p}\sqcup Z\]
where the face maps are induced by the codiagonals and the map $X\to Y$ and $X\to Z$ and the degeneracies are induced by the maps from the empty simplicial set to $X$, $Y$ and $Z$.

For a finite set $S$, and any simplicial space $U_\bullet$, there is an isomorphism
\[|U_\bullet^S|\cong |U_\bullet|^S\]

Therefore, there is a weak equivalence of functors on $\cat{Fin}$
\[\Map(-,P)\simeq |\on{B}_\bullet(\Map(-,Y),\Map(-,X),\Map(-,Z))|\]
where the bar construction on the right hand side is in the category $\on{Fun}(\cat{Fin},\S)$ with the convolution tensor product.

We can form the following bisimplicial object in $\cat{Ch}_{\geq 0}(R)$:
\[\on{B}_\bullet(\on{B}_\bullet(\Map(-,Y),\Map(-,X),\Map(-,Z)),\cat{Fin},A)\]

By the previous observation, if we first realize with respect to the inner simplicial variable and then the outer one, we find something equivalent to $\HH{P}(A|R)$. If we first realize with respect to the outer variable, we find
\[\on{B}_\bullet(\HH{Y}(A|R),\HH{X}(A|R),\HH{Z}(A|R))\]
The two realizations are equivalent which concludes the proof.
\end{proof}

\begin{coro}\label{Hochschild homology}
Let $A$ be a degreewise projective commutative algebra in $\cat{Ch}_{\geq 0}(R)$, then $\HH{S^1}(A)$ is quasi-isomorphic to the Hochschild chains on $A$.
\end{coro}

\begin{proof}
We can write $S^1$ as the homotopy pushout of
\[\xymatrix{
S^0\ar[d]\ar[r]&\on{pt}\\
\on{pt}&}\]
If $S$ is a finite set $\HH{S}(A)=A^{\otimes S}$ with the obvious commutative algebra structure. In particular, the previous theorem gives
\[\HH{S^1}(A)\simeq |\on{B}_\bullet(A,A\otimes A,A)|\]
Since $A=A\op$, the right hand side is quasi-isomorphic to $A\otimes^{\L}_{A\otimes A\op}A$
\end{proof}

\section{The spectral sequence}

We construct a spectral sequence converging to factorization homology. Its $\mathrm{E}^2$-page is identified with Pirashvili's higher Hochschild homology. For $R$ a $\mathbb{Z}$-graded ring, we denote by $\cat{GrMod}_R$ the category of $\mathbb{Z}$-graded left $R$-modules. We denote by $[n]$ the shift by $n$ functor. More precisely, if $M$ is an object of $\cat{GrMod}_R$, $M[n]$ is the graded $R$-module which in degree $k$ is $M_{k-n}$. 

\begin{defi}
Let $\cat{I}$ be a small discrete category and $F:\cat{I}\to \cat{GrMod}_R$ be a functor landing in the category of graded modules over $R$. We define \emph{the homology of $\cat{I}$ with coefficients in $F$} to be the homology groups of the homotopy colimit of $F$ seen as a functor concentrated in homological degree $0$ from $\cat{I}$ to $\cat{Ch}_{\geq 0}(\cat{GrMod}_R)$.

We write $\on{H}^R_*(\cat{I},F)$ for the homology of $\cat{I}$ with coefficients in $F$.
\end{defi}

Note that since we consider graded modules, the chain complexes are graded chain complexes which means that each homology groups is graded. We denote by $\on{H}_{s,t}^R(\cat{I},F)$ the degree $t$ part of the $s$-th homology group. Note that $s$ lives in $\mathbb{Z}_{\geq 0}$ whereas $t$ lives in $\mathbb{Z}$.

There is an explicit model for this homology. We construct the simplicial object of $\cat{GrMod}_R$ whose $p$ simplices are
\[\on{B}_p(R,\cat{I},F)=\bigoplus_{i_0\to\ldots\to i_p}F(i_0)\]

We can form the normalized chain complex associated to this simplicial object in $\cat{GrMod}_R$  and we get a non-negatively graded chain complex in $\cat{GrMod}_R$. Its homology groups are the homology groups of $\cat{I}$ with coefficients in $F$.

Note that if $E$ is an associative algebra in symmetric spectra, then $E_*=\pi_*(E)$ is an associative ring in graded abelian groups and if $M$ is a left $E$-module, then $\pi_*(M)$ is an object of $\cat{GrMod}_{E_*}$.

\begin{prop}
Let $F:\cat{I}\to\Mod_E$ be a functor from a discrete category to the category of left modules over an associative algebra in symmetric spectra $E$. There is a spectral sequence of $E_*$-modules
\[\on{E}^2_{s,t}\cong \on{H}_{s,t}^{E_*}(\cat{I},\pi_*(F))\implies \pi_{s+t}(\on{hocolim}_{\cat{I}}F)\]
\end{prop}

\begin{proof}
The homotopy colimit can be computed by taking an objectwise cofibrant replacement of $F$ and then take the realization of the Bar construction
\[\on{hocolim}_{\cat{I}}F\simeq |\on{B}_\bullet(*,\cat{I},QF(-))|\]
We can then use the standard spectral sequence associated to a simplicial object
\end{proof}

Now assume that $E$ is commutative. Let $A$ be an $\oper{E}_d$-algebra in $\Mod_E$. Let $M$ be a framed $d$-manifold and let $\cat{D}(M)$ be the poset of open sets of $M$ that are diffeomorphic to a disjoint union of copies of $D$. We know from \ref{left} that the factorization homology of $A$ over $M$ can be computed as the homotopy colimit of the composition:
\[\cat{D}(M)\xrightarrow{\delta}\cat{E}_d\xrightarrow{A}\Mod_E\]

We are in a situation where we can apply the previous proposition. We thus get a spectral sequence of $E_*$-modules
\[\on{H}_{s,t}^{E_*}(\cat{D}(M),\pi_*(A\circ\delta))\implies \pi_{s+t}(\int_MA)\]

We want to exploit the fact that $A$ is a monoidal functor to obtain a more explicit model for the left hand side in some cases. 

From now on, $K$ denotes an associative algebra in ring spectra with a $\mathbb{Z}/2$-equivariant Künneth isomorphism. That is, we assume that the obvious map
\[K_*(X)\otimes_{K_*}K_*(Y)\to K_*(X\wedge Y)\]
is an isomorphism of functors of the pair $(X,Y)$ which is equivariant with respect to the obvious $\mathbb{Z}/2$-action on both sides.

Example of such spectra are the Eilenberg-MacLane spectra $Hk$ for any field $k$ or $K(n)$ the Morava $K$-theory of height $n$ at odd primes. 

\begin{prop}
There is a spectral sequence of $K_*(E)$-modules
\[\on{H}_*^{K_*E}(\cat{D}(M),K_*(A\circ\delta))\implies K_*(\int_M A)\]
\end{prop}

\begin{proof}
We just smash the simplicial object computing $\on{hocolim}_{\cat{D}(M)}A(\delta-)$ with $K$ in each degree and take the associated spectral sequence.
\end{proof}

Now we want to identify $K_*(A\circ\delta)$ as a functor on $\cat{D}(M)$.

\begin{prop}
If $d=1$, $K_*(A)$ is an associative algebra in $K_*E$-modules, If $d>1$, $K_*(A)$ is a commutative algebra in the category of $K_*E$-modules. 
\end{prop}

\begin{proof}
If $A$ is an associative (resp. commutative algebra) in $\on{Ho}(\cat{C})$, then $K_*(A)$ is an associative graded $K_*$-module. An $\oper{E}_1$ algebra in $\Mod_E$ is in particular an associative algebra in $\on{Ho}(\cat{C})$ and an $\oper{E}_d$-algebra with $d>1$ is a commutative algebra in $\on{Ho}(\cat{C})$. Thus $K_*(A)$ is an associative (resp. commutative) algebra in $K_*$-modules and the unit map $E\to A$ makes it into an associative (resp. commutative) algebra in $K_*E$-modules. 
\end{proof}

Now, we focus on the case where $d>1$. We have an obvious functor $\alpha:\cat{D}(M)\to\cat{Fin}$ which sends a configuration of disks on $M$ to its set of connected components. In particular, we can consider the functor
\[\cat{D}(M)\xrightarrow{\alpha}\cat{Fin}\goto{K_*(A)}\cat{GrMod}_{K_*E}\]
where the second map is induced by the commutative algebra structure on $K_*(A)$ that we have constructed in the previous proposition. It is clear that this functor coincides with the functor obtained by applying $K_*$ to the composite
\[\cat{D}(M)\goto{\delta}\cat{E}_d\goto{A}\Mod_E\]

From this, we deduce the following proposition :

\begin{prop}\label{spectral sequence}
There is an isomorphism
\[\on{H}_*^{K_*E}(\cat{D}(M),K_*(A\circ\delta))\cong \on{HH}_*^{\on{Sing}(M)}(K_*A|K_*E)\]
In particular, there is a spectral sequence
\[\on{HH}_s^{\on{Sing}(M)}(K_*A|K_*E)_t\implies K_{s+t}(\int_MA)\]
\end{prop}

\begin{proof}
The first claim immediately implies the second. 

In order to prove the first claim, we first observe that we have weak equivalences
\[*\otimes^{\L}_{\cat{D}(M)}K_*(A\circ\delta)\simeq \L\alpha_!*\otimes^{\L}_{\cat{Fin}}K_*(A)\]
where $*$ denotes the constant functor with value $*$.

We have $\L\alpha_!*(S)=\on{hocolim}_{U\in\cat{D}(M)}\cat{Fin}(S,\pi_0(U))$. By \cite[Proposition 5.3.]{horeloperads}, this contravariant functor on $\cat{Fin}$ coincides up to weak equivalences with $S\mapsto \on{Sing}(M)^S$.
\end{proof}

\begin{rem}
The spectral sequence above still exists if $K$ does not have a K\"unneth isomorphism as long as $K_*A$ is flat as a $K_*$-module. We leave the details to the interested reader.
\end{rem}

\subsection*{Multiplicative structure}

Let us start with the general homotopy colimit spectral sequence

\begin{prop}
Let $F:\cat{I}\to \Mod_E$ and $G:\cat{J}\to\Mod_E$ be functors. We have the following equivalence
\[\on{hocolim}_{\cat{I}\times\cat{J}}F\otimes_E G\simeq(\on{hocolim}_\cat{I}F)\otimes_E(\on{hocolim}_{\cat{J}}G)\]
\end{prop}

\begin{proof}
Assume $F$ and $G$ are objectwise cofibrant. The right-hand side is the homotopy colimit over $\Delta\op\times\Delta\op$ of
\[\on{B}_\bullet(*,\cat{I},F)\otimes\on{B}_\bullet(*,\cat{J},G)\]
The diagonal of this bisimplicial object is exactly
\[\on{B}_\bullet(*,\cat{I}\times\cat{J},F\otimes_EG)\]

Since $\Delta\op\to\Delta\op\times\Delta\op$ is homotopy cofinal, we are done.
\end{proof}

We denote by $\on{E}^r_{**}(\cat{I},F)$ the spectral sequence computing the homotopy colimit of $F$.

\begin{prop}
We keep the notations of the previous proposition. There is a pairing of spectral sequences of $E_*$-modules
\[\on{E}^r_{**}(\cat{I},F)\otimes_{E_*} \on{E}^r_{**}(\cat{J},G)\to\on{E}^r_{**}(\cat{I}\times\cat{J},F\otimes_E G)\]
\end{prop}

\begin{proof}
The result is a standard fact about pairing of spectral sequences associated to simplicial objects.
\end{proof}

Let us specialize to the case of factorization homology. We consider an $\oper{E}_d$-algebra $A$ in $\Mod_E$ a homology theory with $\mathbb{Z}/2$-equivariant K\"unneth isomorphism $K$ and a framed manifold of dimension $d$ $M$. We denote by $\on{E}^r_{**}(M,A,K)$ the spectral sequence of the previous section.

\begin{prop}
Let $M$ and $N$ be two framed $d$-manifolds. There is a pairing of spectral sequences
\[\on{E}^r_{**}(M,A,K)\otimes_{K_*E}\on{E}^r_{**}(N,A,K)\to \on{E}^r_{**}(M\sqcup N,A,K)\]
\end{prop}

\begin{proof}
This follows from the previous proposition as well as the observation that $\cat{D}(M\sqcup N)\cong\cat{D}(M)\times\cat{D}(N)$ and the fact that $A\otimes_EA$ as a functor on $\cat{D}(M)\times\cat{D}(N)$ is equivalent to $A$ as a functor on $\cat{D}(M\sqcup N)$.
\end{proof}

In other words, we have proved that the spectral sequence $\on{E}^r_{**}(M,A,K)$ is a lax monoidal functor of the variable $M$. In particular it preserves associative algebras.

Assume now that $M$ is an associative algebra up to isotopy in $f\Man_d$. One possible example is to take $M= N\times\mathbb{R}$ with $N$ a $d$-framed $(d-1)$-manifold. In that case, $M$ is an $\oper{E}_1$-algebra in $f\Man_d$.

\begin{prop}
Let $M$ be an associative algebra up to isotopy of dimension at least $2$. The spectral sequence $\on{E}^r_{**}(M,A,K)$ has a commutative multiplicative structure converging to the associative algebra structure on $K_*\int_{M}A$. 

On the $\on{E}^2$-page, this multiplication is induced by the unique commutative algebra structure on $\on{Sing}(M)$ in the category $(\S,\sqcup)$.

Moreover this structure is functorial with respect to embeddings of $d$-manifolds $M\to M'$ preserving the multiplication up to isotopy.
\end{prop}

\begin{proof}
According to the previous proposition there is a multiplicative structure on the spectral sequence converging to the associative algebra structure on $K_*\int_{M}A$. 

It is easy to see that the multiplication on the $\on{E}^2$-page is what is stated. Since $\on{Sing}(M)$ is commutative, the multiplication on the $\on{E}^2$-page is commutative. The homology of a commutative differential graded algebra is a commutative algebra, therefore the multiplication is commutative on each page.

The functoriality is clear.
\end{proof}

Now we want to construct an edge homomorphism

Let $S$ be a $(d-1)$-manifold with a $d$-framing $\tau$. Let $\phi$ be a framed embedding of $\mathbb{R}^{d-1}\times\mathbb{R}$ into $S\times{\mathbb{R}}$ commuting with the projection to $\mathbb{R}$. Applying factorization homology we get a map of $\oper{E}_1$-algebras:
\[u_\phi:A\cong \int_{\mathbb{R}^{d-1}\times\mathbb{R}} A\to\int_{S\times\mathbb{R}}A\]

On the other hand for any point $x$ of $S\times\mathbb{R}$, we get a morphism of commutative algebra over $K_*E$:
\[u_x:K_*(A)\cong \on{HH}^{\on{pt}}(K_*A|K_*E)\to\on{HH}^{\on{Sing}(S)}(K_*A|K_*E)\]

\begin{prop}
For any framed embedding $\phi:\mathbb{R}^{d-1}\times\mathbb{R}\to S\times\mathbb{R}$, there is an edge homomorphism
\[K_*A\to \on{E}^r_{0,*}(S\times\mathbb{R},A,K)\]
On the $\on{E}^2$-page it is identified with the $K_*E$-algebra homomorphism
\[u_{\phi(0,0)}:K_*(A)\to\on{HH}^{\on{pt}}(K_*A|K_*E)\to\on{HH}^{\on{Sing}(S)}(K_*A|K_*E)\]
and it converges to the $K_*E$-algebra homomorphism
\[K_*(u_\phi):K_*A\to K_*\int_{N\times\mathbb{R}}A\]
\end{prop}

\begin{proof}
The spectral sequence computing $K_*\int_{\mathbb{R}^{d-1}\times\mathbb{R}}A$ has its $\on{E}^2$-page $K_*A$ concentrated on the $0$-th column. For degree reason, it degenerates. Then the result follows directly from the functoriality of the spectral sequence applied to the map $\phi$.
\end{proof}

Note that the edge homomorphism only depends on the connected component of the image of $\phi$.

In the case of the sphere $S^{d-1}\times\mathbb{R}$ with the framing $\kappa$, we can say more:

\begin{lemm}\label{splitting}
For any framed embedding $\phi:\mathbb{R}^{d-1}\times\mathbb{R}\to (S^{d-1}\times\mathbb{R})_\kappa$ commuting with the projection to $\mathbb{R}$, the map
\[u_\phi:A\to\int_{S^{d-1}\times\mathbb{R}} A\]
has a section in the homotopy category of $\Mod_E$
\end{lemm}

\begin{proof}
There is an embedding:
\[S^{d-1}\times\mathbb{R}\to\mathbb{R}^d\]
sending $(\theta,x)$ to $e^x\theta$. This embedding preserves the framing. Moreover, the composite:
\[\mathbb{R}^d\stackrel{\phi}{\rightarrow}S^{d-1}\times\mathbb{R}\to\mathbb{R}^d\]
is isotopic to the identity (because $\Emb_f(\mathbb{R}^d,\mathbb{R}^d)$ is contractible). We can apply $\int_{-}A$ to this sequence of morphisms of framed manifolds and we obtain the desired splitting.
\end{proof}

Although we will not need it, this has the following corollary:

\begin{coro}
The image of the edge homomorphism in $\on{E}^r_{**}((S^{d-1}\times\mathbb{R})_\kappa,A,K)$ consists of permanent cycles.
\end{coro}

\begin{rem}
Our geometric description of higher Hochschild cohomology in \ref{Hom over cylinder} can be used to construct a similar spectral sequence calculating $K_*\HC{\oper{E}_d}(A)$ and whose $\on{E}_2$-page is a cohomological version of higher Hochschild cohomology defined in \cite{ginothigher}. However, this spectral sequence does not always converge.
\end{rem}

\section{Computations}

\begin{prop}\label{etale descent}
Let $A_*$ be a degreewise projective commutative graded algebra over a commutative graded ring $R_*$. Assume that $A_*$ is a filtered colimit of \'etale algebras over $R_*$. Then,  for all $d\geq 1$, the unit map
\[A_*\to\on{HH}^{S^d}(A_*|R_*)\]
is a quasi isomorphism of commutative $R_*$-algebras.
\end{prop}

\begin{proof}
We proceed by induction on $d$. For $d=1$, $\on{HH}^{S^1}(A_*|R_*)$ is quasi-isomorphic to the ordinary Hochschild homology $\HH{}(A_*|R_*)$  by proposition \ref{Hochschild homology}. If $A_*$ is \'etale, the result is well-known (see for instance \cite[\'Etale descent theorem p. 368]{weibeletale}). If $A_*$ is a filtered colimit of \'etale algebras, the result follows from the fact that Hochschild homology commutes with filtered colimits.

Now assume that $A_*\to\HH{S^{d-1}}(A_*|R_*)$ is a quasi-isomorphism of commutative algebras. The sphere $S^d$ is part of the following homotopy pushout diagram
\[\xymatrix{
S^{d-1}\ar[r]\ar[d]&\on{pt}\ar[d]\\
\on{pt}\ar[r]& S^{d}
}
\]
Applying \ref{HH pushout}, we find
\[\HH{S^d}(A_*|R)\simeq |\on{B}_\bullet(A_*,\HH{S^{d-1}}(A_*|R_*),A_*)|\]
The quasi-isomorphism $A_*\to\HH{S^{d-1}}(A_*|R_*)$ induces a degreewise quasi-isomorphism between Reedy cofibrant simplicial objects
\[\on{B}_\bullet(A_*,A_*,A_*)\to\on{B}_\bullet(A_*,\HH{S^{d-1}}(A_*|R_*),A_*)\]
This induces a quasi-isomorphism between their realization
\[A_*\simeq \HH{S^d}(A_*|R_*)\]
\end{proof}

\begin{coro}\label{unit equivalence}
Let $A$ be an $\oper{E}_d$-algebra in $\cat{C}$ such that $K_*(A)$ is a filtered colimits of \'etale algebras over $K_*$, then the unit map:
\[A\to\int_{S^{d-1}\times\mathbb{R}}A\]
is a $K$-local equivalence.
\end{coro}

\begin{proof}
It suffices to check that the $K$-homology of this map is an isomorphism. This can be computed as the edge homomorphism of the spectral sequence $\on{E}^2(S^{d-1}\times\mathbb{R},A,K)$. By the previous proposition, the edge homomorphism is an isomorphism on the $\on{E}^2$-page. Therefore, the spectral sequence collapses at the $\on{E}^2$-page for degree reasons.
\end{proof}

Let us fix a prime $p$. We denote by $E_n$, the Lubin-Tate ring spectrum of height $n$ at $p$ and $K_n$ the $2$-periodic Morava $K$-theory of height $n$. Recall that
\begin{align*}
(E_n)_* &\cong\mathbb{W}(\mathbb{F}_{p^n})[[u_1,\ldots,u_{n-1}]][u^{\pm 1} ],\;|u_i|=0\;|u|=2\\
(K_n)_* &\cong \mathbb{F}_{p^n}[u^{\pm 1}]=(E_n)_*/(p,u_1,\ldots,u_{n-1})
\end{align*}

The spectrum $E_n$ is known to have a unique $\oper{E}_1$-structure inducing the correct multiplication on homotopy groups (this is a theorem of Hopkins and Miller, see \cite{rezknotes}) and a unique commutative structure (see \cite[Corollary 7.6.]{goerssmoduli}). As far as we know, there is no published proof that the space of $\oper{E}_d$-structure for $d\geq 2$ is contractible although evidence suggests that it is the case.

The spectrum $K_n$ has a $\mathbb{Z}/2$-equivariant K\"unneth isomorphism if $p$ is odd. If $p=2$, the equivariance is not satisfied in general but it is true if we restrict $(K_n)_*$ to spectra whose $K_n$-homology is concentrated in even degree like $E_n$. Our argument works at $p=2$ modulo this minor modification.

\begin{coro}\label{LT is etale}
For any positive integer $n$, and any $\oper{E}_d$-algebra structure on $E_n$ inducing the correct multiplication on homotopy groups, the unit map
\[E_n\to\int_{S^{d-1}\times\mathbb{R}}E_n\]
induces an isomorphism in $K_n$-homology.
\end{coro}

\begin{proof}
By \cite[Corollary 4.10]{hoveyoperations}, for any such $\oper{E}_d$-structure on $E$, we have 
\[(K_n)_*(E_n)\cong C(\Gamma,(K_n)_*)\]
where the right hand side denotes the set of continuous maps $\Gamma\to (K_n)_*$ where $\Gamma$ is the Morava stabilizer group with its profinite topology and $(K_n)_*$ is given the discrete topology. By definition of a profinite group, the group $\Gamma$ is an inverse limit $\Gamma=\on{lim}_U\Gamma/U$ taken over the filtered poset of open finite index subgroups $U$ of $\Gamma$. Thus, we have
\[C(\Gamma,(K_n)_*)=\on{colim}_UC(\Gamma/U,(K_n)_*)\]
This expresses $(K_n)_*E_n$ as a filtered colimit of \'etale algebras over $(K_n)_*$. Thus by \ref{unit equivalence}, we get the desired result.
\end{proof}

\begin{prop}\label{E_n hochschild cohomology}
Same notations, the map $\HC{\oper{E}_d}(E_n)\to E_n$ is an equivalence.
\end{prop}

\begin{proof}
We have
\[\HC{\oper{E}_d}(E_n)\simeq\R\Hom_{\int_{S^{d-1}\times\mathbb{R}}E_n}(E_n,E_n)\]
This can be computed as the end
\[\hom_{\Disk^{-+}}(\Emb^{S^0}(-,[0,1]),\mathscr{F}(E_n,\int_{S^{d-1}\times\mathbb{R}}E_n,E_n))\]

The spectrum $E_n$ is $K(n)$-local, therefore, $\Hom(-,E_n)$ sends $K(n)$-equivalences to equivalences. This implies that
\[\mathscr{F}(E_n,\int_{S^{d-1}\times\mathbb{R}}E_n,E_n)\simeq \mathscr{F}(E_n,E_n,E_n)\]

Therefore, we have
\[\HC{\oper{E}_d}(E_n)\simeq\R\Hom_{E_n}(E_n,E_n)\]
\end{proof}

We can prove a variant of the previous result. Let $E(n)=BP/(v_{n+1},v_{n+2},\ldots)[v_n^{-1}]$ be the Johnson-Wilson spectrum, let $K(n)$ be the $v_n$ periodic Morava $K$-theory with $K(n)_*=E(n)/(p,v_1,\ldots,v_{n-1})=\mathbb{F}_p[v_n^{\pm 1}]$. Let $\widehat{E}(n)$ be $L_{K(n)}E(n)$.

\begin{prop}
For any $\oper{E}_d$-algebra structure on $\widehat{E}(n)$ inducing the correct multiplication on homotopy groups, the action map
\[\HC{\oper{E}_d}(\widehat{E}(n))\to\widehat{E}(n)\]
is a weak equivalence.
\end{prop}

\begin{proof}
The proof is exactly the same once we know that $K(n)_*\widehat{E}(n)$ is the commutative ring
\[K(n)_*\widehat{E}(n)=C(\Gamma,K(n)_*)\]
where $\Gamma$ is again the Morava stabilizer group. 
\end{proof}

\section{\'Etale base change for Hochschild cohomology}

In this section we put the previous result in the wider context of derived algebraic geometry over $\oper{E}_d$-algebra. This section is inspired by \cite{francistangent}. 

We let $(\cat{C},\otimes,\un)$ denote the category $\Mod_E$ but the arguments hold more generally. Note however that we need $\cat{C}$ to be stable in this section.

There is a ``polar coordinate'' embedding $S^{d-1}\times(0,1)\to D$ sending $(\theta,r)$ to $e^{r-1}\theta$.

\begin{defi}
Let $A$ be an $\oper{E}_d$-algebra in $\cat{C}$. The cotangent complex $L_A$ of $A$ is defined to be the $n$-fold desuspension of the cofiber of the map
\[\int_{S^{d-1}\times\mathbb{R}}A\to \int_{\mathbb{R^d}}A \cong A\]
induced by the polar coordinate embedding.
\end{defi}

\begin{prop}
This coincides with the cotangent complex of $A$ defined by Francis.
\end{prop}

\begin{proof}
Both sides of the map commutes with homotopy colimits of $\oper{E}_d$-algebras, therefore it suffices to check it for free $\oper{E}_d$-algebras. Let $A=F_{\oper{E}_d}(V)$. We can use \cite[Proposition 5.8]{francisfactorization}, we see that
\[\int_{S^{d-1}\times(0,1)}F_{\oper{E}_d}(V)\simeq \bigsqcup_{i\geq 0}\on{Conf}(i,S^{d-1}\times(0,1))\otimes_{\Sigma_i}V^{\otimes i}\]
and similarly
\[\int_{D}F_{\oper{E}_d}(V)\simeq \bigsqcup_{i\geq 0}\on{Conf}(i,D)\otimes_{\Sigma_i}V^{\otimes i}\]
On the other hand, it is proved in \cite[Theorem 2.26]{francistangent}  that there is a cofiber sequence
\[\int_{S^{d-1}\times(0,1)}A\to A\to L_A[n]\]
moreover, the proof of \cite[Theorem 2.26]{francistangent} is based on an explicit computation in the free case and it is easy to see by looking at this proof that the first map in the above cofiber sequence coincides with the ``polar embedding'' map.
\end{proof}

\begin{rem}
The above definition is a bit ad hoc. Francis actually defines in \cite[Definition 2.10]{francistangent} the cotangent complex as the object representing the derivations:
\[\R\Hom_{S_{\kappa}^{d-1}\Mod_A}(L_A,M)\simeq \R\Hom_{\cat{C}[\oper{E}_d]/A}(A,A\oplus M):=\on{Der}(A,M)\]
The fact that the two definitions coincide is then \cite[Theorem 2.26]{francistangent}.
\end{rem}

\begin{defi}
We say that an $\oper{E}_d$-algebra $A$ is \'etale if $L_A$ is contractible. More generally, given an object $Z$ in $\cat{C}$, we say that $A$ is $Z$-locally \'etale if $Z\otimes L_A$ is contractible.
\end{defi}

We say that a a map $X\to Y$ in $\cat{C}$ is a $Z$-local weak equivalence if the induced map $X\otimes^\L Z\to Y\otimes^\L Z$ is a weak equivalence.

An equivalent formulation of the previous definition is that $A$ is ($Z$-locally) \'etale if the unit map $A\to \int_{S^{d-1}\times(0,1)}A$ is a ($Z$-local) equivalence. Indeed we have shown in \ref{splitting} that the unit map is a section of $\int_{S^{d-1}\times(0,1)}A\to A$.

\begin{prop}
If $A$ is a commutative algebra and is ($Z$-locally) \'etale as an $\oper{E}_d$-algebra, then it is ($Z$-locally) \'etale as an $\oper{E}_{d+1}$-algebra.
\end{prop}

\begin{proof}
We have proved in \cite[Theorem 5.8.]{horeloperads} that for a commutative algebra $A$, $\int_M A$ is equivalent to $\on{Sing}(M)\otimes A$ (i.e. the tensor in the category of commutative algebras in $\Mod_E$). Then the proof is exactly the proof of \ref{etale descent}.
\end{proof}

\begin{rem}
More generally using the excision property for factorization homology (see \cite[Lemma 3.11.]{francisfactorization}), one can prove that if $A$ is $\oper{E}_{d+1}$ and is  ($Z$-locally) \'etale as an $\oper{E}_d$-algebra, it is ($Z$-locally) \'etale as an $\oper{E}_{d+1}$-algebra.
\end{rem}

\begin{rem}
If $A$ is a commutative algebra, then $A$ is \'etale as an $\oper{E}_2$-algebra if and only if it is formally $\on{THH}$-\'etale (i.e. if the map $A\to\on{THH}(A)$ is an equivalence).  Indeed, for commutative algebras (and in fact for an $\oper{E}_3$-algebras), $\on{THH}(A)$ coincides with $\int_{S^1\times\mathbb{R}}A$. Note that is is \emph{not} true for $\oper{E}_2$-algebras as the product framing on $S^1\times\mathbb{R}$ is not connected to the $\kappa$-framing in the space of framings of $S^1\times\mathbb{R}$.
\end{rem}

\begin{rem}
If $A$ is a commutative algebra, $\int_{S^{d-1}\times(0,1)}A\simeq S^{d-1}\otimes A$ by \cite[Theorem 5.8.]{horelfactorization}. Therefore, $A$ is \'etale as an $\oper{E}_{d+1}$-algebra if and only if the space $\Map_{\Mod_E[\oCom]}(A,B)$ is $d$-truncated for any $B$.
\end{rem}

Recall that an object $U$ of $\cat{C}$ is said to be $Z$-local if for all $Z$-local weak equivalence $X\to Y$, the induced map
\[\R\Hom(Y,U)\to\R\Hom(X,U)\]
is a weak equivalence in $\cat{C}$.

\begin{lemm}
Let $u:R\to S$ be a map of cofibrant associative algebras in $\cat{C}$ that is a $Z$-local weak equivalence and $M$ and $N$ be two left modules over $S$ with $N$ $Z$-local in $\cat{C}$. Then the map 
\[\R\Hom_{\Mod_S}(M,N)\to\R\Hom_{\Mod_R}(u^*M,u^*N)\]
is a weak equivalence.
\end{lemm}

\begin{proof}
After maybe taking a cofibrant replacement of $M$ and a fibrant replacement of $N$, the left hand side can be computed as the homotopy limit of the cobar construction
\[[n]\mapsto \Hom(S^{\otimes n}\otimes M,N)\]
Similarly, the left hand side can be computed as the homotopy limit of
\[[n]\mapsto\Hom(R^{\otimes n}\otimes M,N)\]
Since $R\to S$ is a $Z$-local weak equivalence so is $R^{\otimes n}\otimes M\to S^{\otimes n}\otimes M$ for each $n$. Thus, since $N$ is $Z$-local, the two cosimplicial objects are weakly equivalent, which implies that they have weakly equivalent homotopy limits.
\end{proof}

We can now state and prove the main theorem of this section. 

\begin{theo}\label{etale base change for Hochschild cohomology}
Let $T$ be a commutative algebra in $\cat{C}$ that is ($Z$-locally) \'etale as an $\oper{E}_d$-algebra over $\un$, then for any $\oper{E}_d$-algebra $A$ over $T$ (which is $Z$-local as an object of $\cat{C}$) the base-change map 
\[\HC{\oper{E}_d}(A)\to\HC{\oper{E}_d}(A|T)\]
is a weak equivalence.
\end{theo}

\begin{proof}
We write $A|T$ whenever we want to emphasize the fact that we are seeing $A$ as an $\oper{E}_d$-algebra over $T$. 

By Francis (\cite[Proposition 2.11]{francistangent}), there is cofiber sequence
\[u_!L_T\to L_A\to L_{A|T}\]
where $u:T\to A$ is the unit map and $u_!$ is the corresponding functor
\[u_!:S_{\kappa}^{d-1}\Mod_T\to S_{\kappa}^{d-1}\Mod_A\]
By hypothesis $L_T$ is ($Z$-locally) contractible, therefore $L_A\to L_{A|T}$ is a ($Z$-local) equivalence.

We have a base-change map of cofiber sequences
\[\xymatrix{
\Sigma^{d-1}L_A\ar[d]\ar[r]&\int_{S^{d-1}\times(0,1)}A\ar[d]\ar[r]&A\ar[d]^{\id}\ar[r]&\Sigma^dL_A\ar[d]\\
\Sigma^{d-1}L_{A|T}\ar[r]&\int_{S^{d-1}\times(0,1)}A|T\ar[r]&A\ar[r]&\Sigma^dL_{A|T}
}\]
This implies that $\int_{S^{d-1}\times(0,1)}A\to \int_{S^{d-1}\times(0,1)}A|T$ is a ($Z$-local) equivalence.

We can form the commutative diagram
\[
\xymatrix{
U_A^{S^{d-1}_{\kappa}}\ar[d]\ar[r]& \int_{S^{d-1}\times(0,1)}A\ar[d]\\
U^{S^{d-1}_{\kappa}}_{A|T}\ar[r]&\int_{S^{d-1}\times(0,1)}A|T
}
\]
where the horizontal maps are the maps of corollary \ref{coro-equivalence enveloping factorization}. These maps are weak equivalences by \ref{coro-equivalence enveloping factorization}. Thus, the map $U_A^{S^{d-1}_\kappa}\to U_{A|T}^{S^{d-1}_\kappa}$ is a ($Z$-local) weak equivalence of associative algebras. The theorem follows from this fact and the previous lemma.
\end{proof}

\begin{rem}
The computation of the previous section shows that $S\to E_n$ is $K(n)$-locally an \'etale morphism of $\oper{E}_d$-algebras for all $d$. Therefore, given a $K(n)$-local $E_n$-algebra $A$, one can compute its (higher) Hochschild cohomology over $E_n$ or over $S$ without affecting the result. This fact is used by Angeltveit (see \cite[Theorem 6.9.]{angeltveittopological}) in the case of ordinary Hochschild cohomology. 
\end{rem}

\bibliographystyle{alpha}
\bibliography{biblio}

\end{document}